\documentclass[11pt, a4paper]{amsart}
\usepackage[utf8]{inputenc}
\usepackage{amsmath, amsfonts, amssymb, amsthm}
\usepackage{hyperref}
\usepackage[inline]{enumitem}
\usepackage{multicol}
\usepackage{tikz, float, tikz-cd}
	\usetikzlibrary{arrows}
\theoremstyle{definition}
\newtheorem{definition}{Definition}[section]

\newtheorem*{acknowledgments}{Acknowledgments}

\theoremstyle{plain}
\newtheorem{proposition}[definition]{Proposition}
\newtheorem{lemma}[definition]{Lemma}
\newtheorem{theorem}[definition]{Theorem}

\newtheorem{problem}[definition]{Problem}
\newtheorem{convention}{Convention}

\usepackage[
    a4paper,
    top=2.5cm,
    bottom=2.5cm,
    left=2.5cm,
    right=2.5cm,
    includeheadfoot,
    headheight=14pt,
    footskip=1cm,
    marginparwidth=3cm,
    marginparsep=0.5cm
]{geometry}

\hypersetup{pdfborder={0 0 0}}

\newcommand{\Z}{\mathbb{Z}}
\newcommand{\Q}{\mathbb{Q}}

\newcommand{\K}{\mathbf{k}} 
\newcommand{\V}{\mathcal{V}} 
\newcommand{\W}{\mathcal{W}} 
\newcommand{\gr}{\mathrm{gr}} 
\newcommand{\F}{\mathcal{F}} 
\newcommand{\Betti}{\mathrm{B}} 
\newcommand{\DeRham}{\mathrm{DR}} 
\newlist{termlist}{enumerate}{1}
\setlist[termlist]{label=\textbf{Term \arabic*}, align=left, wide=1pt}

\author{Benjamin Enriquez}
\address{\scriptsize Institut de Recherche Mathématique Avancée (UMR 7501), Université de Strasbourg, 7 rue René Descartes, 67000 Strasbourg, France}
\email{b.enriquez@math.unistra.fr}
\author{Khalef Yaddaden}
\address{\scriptsize Graduate School of Mathematics, Nagoya University, Furo-cho, Chikusa-ku, Nagoya, 464-8602, Japan.}
\email{khalef.yaddaden.c8@math.nagoya-u.ac.jp}

\title[Compatibility of the Betti harmonic coproduct with cyclotomic filtrations]{On the compatibility of the Betti harmonic coproduct with cyclotomic filtrations}

\date{August 17, 2025.}


\begin{document}
    \begin{abstract}
        In \cite{Yad2}, the second author introduced a Betti counterpart of $N$-cyclotomic double shuffle theory for any $N \geq 1$. The construction is based on the group algebra of the free group $F_2$, endowed with a filtration relative to a morphism $F_2 \to \mu_N$ (where $\mu_N$ is the group of $N$-th roots of unity). One of the main results of \cite{Yad2} is the construction of a complete Hopf algebra coproduct $\widehat{\Delta}^{\mathcal{W}, \mathrm{B}}_N$ on the relative completion of a specific subalgebra $\mathcal{W}^\mathrm{B}$ of the group algebra of $F_2$. However, an explicit formula for this coproduct is missing.
        In this paper, we show that the discrete Betti harmonic coproduct $\Delta^{\mathcal{W}, \mathrm{B}}$ defined in \cite{EF1} for the classical case ($N=1$) by the first author and Furusho remains compatible with the filtration structure on $\mathcal{W}^\mathrm{B}$ induced by the relative completion for arbitrary $N$. This compatibility suggests that the completion corresponding to $\Delta^{\mathcal{W}, \mathrm{B}}$ is a candidate for an explicit realization of $\widehat{\Delta}^{\mathcal{W}, \mathrm{B}}_N$.
    \end{abstract}
    
    \maketitle
	
    {\footnotesize \tableofcontents}

    \section{Introduction}
    Throughout this paper, let $\K$ be a commutative $\Q$-algebra and $N$ be a positive integer. Denote by $\mu_N$ the group of complex $N$-th roots of unity with generator $\zeta_N := e^{\frac{\mathrm{i}2\pi}{N}}$. We will also use the following convention
    \begin{convention} \label{convention_star}
        For $\K$-submodules $A_1, \dots, A_k$ of a $\K$-algebra $A$ and positive integers $n_1, \dots, n_k$, we denote by $A_1^{n_1} \cdots A_k^{n_k}$ the image of the morphism $A_1^{\otimes n_1} \otimes \cdots \otimes A_k^{\otimes n_k} \to A$ induced by the product in $A$. In the expression $A_1^{n_1} \cdots A_k^{n_k}$, we write $A_j$ instead of $A_j^{n_j}$ whenever $n_i=1$ ($1 \leq j \leq k$).
    \end{convention}
    \subsection{Context and motivation}
    Cyclotomic multiple zeta values (CMZVs) are special values of multiple polylogarithms evaluated at roots of unity, defined by the convergent series:
    \begin{equation*}
        \mathrm{Li}_{(k_1, \dots, k_r)}(z_1, \dots, z_r) := \sum_{m_1 > \dots > m_r > 0} \frac{z_1^{m_1} \cdots z_r^{m_r}}{m_1^{k_1} \cdots m_r^{k_r}},
    \end{equation*}
    where \( r, k_1, \dots, k_r \in \Z_{>0} \) and \( z_1, \dots, z_r \in \mu_N \) with \( (k_1, z_1) \neq (1,1) \). These values arise as periods of the motivic fundamental groupoid of the cyclotomic punctured projective line $\mathbb{P}^1 \setminus \{0, \mu_N, \infty\}$ \cite{Del10, Gon05} and are related to associators, mixed Tate motives, and the Grothendieck-Teichmüller group. \newline
    From this perspective, the double shuffle relations among CMZVs --arising from series expansions and iterated integrals-- are encoded in Racinet's ``\emph{double mélange et régularisation} scheme'' $\mathsf{DMR}_N$ \cite{Rac02}, which is expressed in terms of a graded algebra $\V^\DeRham_N$, a graded subalgebra $\W^\DeRham_N$ of $\V^\DeRham_N$, and a Hopf algebra coproduct $\Delta^{\W, \DeRham}_N$ on $\W^\DeRham_N$. More specifically, this framework is built on the completions of these graded objects, that is, the topological algebra $\widehat{\V}^\DeRham_N$ and the complete Hopf algebra $(\widehat{\W}^\DeRham_N, \widehat{\Delta}^{\W, \DeRham}_N)$ \cite{Rac02, Yad1}.
    \\[2mm]
    A Betti analogue of this setting was developed by the second author in \cite{Yad2}, generalizing the work of the first author and Furusho in \cite{EF1} (for $N=1$), which in turn is inspired by the unpublished preprint of Deligne and Terasoma \cite{DeT}. Here, the key objects are a filtered algebra \( \V^\Betti_N \) and subalgebra \( \W^\Betti_N \) of \( \V^\Betti_N \); and the completion \( \widehat{\W}^\Betti_N \), equipped with a complete Hopf algebra coproduct \( \widehat{\Delta}^{\W, \Betti}_N \) --called the $N$-cyclotomic Betti harmonic coproduct-- whose defining property is the conjugation formula \cite[Theorem 3.2.4]{Yad2}
    \begin{equation}\label{eq:conjugation}
        \widehat{\Delta}^{\W, \Betti}_N = (\mathrm{comp}^\W_{\Phi,N} \otimes \mathrm{comp}^\W_{\Phi,N})^{-1} \circ \widehat{\Delta}^{\W, \DeRham}_N \circ \mathrm{comp}^\W_{\Phi,N},
    \end{equation}
    which is valid for any choice of $\Phi \in \mathsf{DMR}_N$; where \( \mathrm{comp}^\W_{\Phi,N} : \widehat{\W}^\Betti_N \to \widehat{\W}^\DeRham_N \) is a comparison isomorphism \cite[Proposition-Definition 3.2.2]{Yad2} attached to $\Phi$.
    \\[2mm]
    For \( N=1 \), a Hopf algebra coproduct $\Delta^{\W, \Betti}$ on $\W^\Betti_1 = \W^\Betti$ was explicitly constructed in \cite{EF1, EF2}, the compatibilty of $\Delta^{\W, \Betti}$ with the filtration on $\W^\Betti$ for $N=1$ was proved, and the corresponding completed coproduct $\widehat{\Delta}^{\W, \Betti}$ was identified with $\widehat{\Delta}^{\W, \Betti}_1$ from \eqref{eq:conjugation}, hence $\widehat{\Delta}^{\W, \Betti}_1 = \widehat{\Delta}^{\W, \Betti}$.
    However, for general $N$, an explicit formula for $\widehat{\Delta}^{\W, \Betti}_N$ is still unknown.
    \subsection{The main results}
    Let $F_2$ be the free group generated by two elements denoted $X_0$ and $X_1$. Consider the group morphism $F_2 \to \mu_N$ given by
    \[
        X_0 \mapsto \zeta_N \text{ and } X_1 \mapsto 1.
    \]
    Its kernel is the group freely generated by the $N+1$ elements \cite[Lemma 3.1.1]{Yad2} 
    \[
        X_0^N  \text{ and } X_0^a X_1 X_0^{-a}, \text{ for } a \in [\![0, N-1]\!].
    \]
    Denote by $\mathcal{I}_N := \ker(\K F_2 \to \K\mu_N)$ where $\K F_2 \to \K\mu_N$ is the $\K$-algebra morphism induced from the group morphism $F_2 \to \mu_N$.
    \begin{definition}[{\cite[Proposition-Definition 3.1.4]{Yad2}}]
        Let $\V^\Betti_N$ be the group algebra $\K{F_2}$ equipped with the algebra filtration given by
        \begin{equation*}
            \mathcal{F}^m \V^{\Betti}_N :=
            \begin{cases}
                \K{F_2} & \text{if } m \leq 0 \\
                \mathcal{I}_N^m & \text{if } m > 0
            \end{cases}, 
        \end{equation*}
        where $\mathcal{I}_N^m$ is the $m$-th power of the ideal $\mathcal{I}_N$ (see Convention \ref{convention_star}).
    \end{definition}

    \begin{definition}[{\cite[Proposition-Definition 3.1.13]{Yad2}}] \label{def:WBFiltration}
        Consider the subalgebra $\W^{\Betti}_N$ of $\V^{\Betti}_N$ given by 
        \[
            \W^{\Betti}_N := \K \oplus \V^{\Betti}_N (X_1 - 1).
        \]
        It is endowed with the algebra filtration given by
        \[  
            \mathcal{F}^m \W^{\Betti}_N := \W^{\Betti}_N \cap \mathcal{F}^m \V^{\Betti}_N, \ \forall m \in \Z.
        \]
    \end{definition}
    
    \noindent When $N=1$, the filtration $(\F^m\V^\Betti_1)_{m \in \Z}$ is the natural filtration of the group algebra $\K{F_2}$ given by powers of the augmentation ideal. Therefore, the induced filtration on $\W^\Betti_1$ corresponds the one given in \cite[Sec. 2.1]{EF1}. We will use the notation $\V^\Betti$ (resp. $\W^\Betti$) instead of $\V^\Betti_1$ (resp. $\W^\Betti_1$) to refer to these naturally filtered algebras.\\[2mm]
    It follows from \cite[Proposition 2.3]{EF1} that the algebra $\W^\Betti$ is generated by 
    \begin{equation*}
        X_1^{-1} \text{ and } X_0^n(X_1-1) \text{ for } n \in \Z. 
    \end{equation*}
    The algebra $\W^\Betti$ is equipped with a bialgebra structure whose coproduct is the algebra morphism $\Delta^{\W, \Betti} : \W^{\Betti} \to \W^{\Betti} \otimes \W^{\Betti}$ given by (see \cite[Lemma 2.11]{EF1})
    \[
        \Delta^{\W, \Betti}(X_1^{-1}) = X_1^{-1} Y_1^{-1},
    \]
    and for $n \in \Z$,
    \[
        \Delta^{\W, \Betti}(X_0^n(X_1-1)) = X_0^n(X_1-1) + Y_0^n(Y_1-1) - \sum_{k=1}^{n-1} X_0^k(X_1-1) Y_0^{n-k}(Y_1-1),
    \]
    where one sets $X_i^{\pm 1} := X_i^{\pm 1} \otimes 1$ and $Y_i^{\pm 1} := 1 \otimes X_i^{\pm 1}$ for $i \in \{0, 1\}$, and one uses the convention that for a map $f$ from $\Z$ to an abelian group and $p, q \in \Z$,
    \begin{equation*}
        \sum_{k=p}^q f(k) :=
        \begin{cases}
            f(p) + \cdots + f(q) & \text{if } q > p - 1 \\
            0 & \text{if } q = p - 1 \\
            -f(p-1) - \cdots - f(q+1) & \text{if } q < p - 1
        \end{cases}
    \end{equation*}
    
    \noindent The following result is the first main theorem of the paper. It states that the coproduct $\Delta^{\W, \Betti}$ is actually compatible with the filtration given in Definition \ref{def:WBFiltration}:
    \begin{theorem} \label{thm:main_theorem1}
        For any $m \in \Z$, we have
        \[
            \Delta^{\W, \Betti}(\F^m\W^\Betti_N) \subset \F^m(\W^\Betti_N \otimes \W^\Betti_N).
        \]
    \end{theorem}

    \begin{definition}[{\cite[§2.1.1]{Yad1}}]
        Let $\V^\DeRham_N$ be the graded $\K$-algebra\footnote{in \cite[§2.1.1]{Yad1} this corresponds to $\V_G$ for $G=\mu_N$.} generated by \(\{e_0, e_1\} \sqcup \mu_N\) where $e_0$ and $e_1$ are of degree $1$ and elements $\zeta \in \mu_N$ are of degree $0$ satisfying the relations:
        \setlength\multicolsep{5pt}
        \begin{multicols}{3}
            \begin{enumerate}[leftmargin=*, label=(\roman*)]
                \item $\zeta \cdot \eta = \zeta \eta$;
                \item $1_{\V^\DeRham_N} = 1$;
                \item $\zeta \cdot e_0 = e_0 \cdot \zeta$;
            \end{enumerate}
        \end{multicols}
        \setlength\multicolsep{5pt}
        \noindent for any $\zeta, \eta \in \mu_N$; where “$\cdot$” is the algebra multiplication\footnote{which we will omit if there is no risk of ambiguity.}.
    \end{definition}

    \noindent Recall from \cite[§2.1.1]{Yad1} the subalgebra
    \[
        \W^\DeRham_N := \K \oplus \V^\DeRham_N e_1
    \]
    of $\V^\DeRham_N$. It is a graded algebra freely generated by (\cite[Proposition 2.6(ii)]{Yad1})
    \[
        Z := \{ z_{n,\zeta} := - e_0^{n-1} \zeta e_1 \, | \, (n, \zeta) \in \Z_{>0} \times \mu_N \},
    \]
    where for any $(n, \zeta) \in \Z_{>0} \times \mu_N$ the element $z_{n,\zeta}$ is of degree $n$.
    Moreover, $\W^\DeRham_N$ is equipped with a Hopf algebra structure with respect to the \emph{harmonic coproduct}, which is the algebra morphism $\Delta^{\W, \DeRham}_N : \W^\DeRham_N \to \W^\DeRham_N \otimes \W^\DeRham_N$ given by (\cite[Proposition 2.11(i)]{Yad1})
    \begin{equation*}
        \Delta^{\W, \DeRham}_N (z_{n, \zeta}) = z_{n, \zeta} \otimes 1 + 1 \otimes z_{n, \zeta} + \sum_{\substack{k=1 \\ \eta \in \mu_N}}^{n-1} z_{k, \eta} \otimes z_{n-k, \zeta \eta^{-1}}. 
    \end{equation*}
    
    \noindent Let $\gr(\V^\Betti_N)$ be the associated graded algebra of $\V^\Betti_N$ for the $\mu_N$-filtration $(\F^m\V^\Betti_N)_{m \in \Z}$. For $m \in \Z$ and $v \in \F^m\V^\Betti_N$, denote by $[v]_m$ the image in $\F^m \V^\Betti_N \big/ \F^{m+1} \V^\Betti_N$ of the element $v$.
    \begin{proposition}[{\cite[Theorem 3.1.6]{Yad} and \cite[Proposition 3.1.12]{Yad}}] \ 
        \begin{enumerate}[label=(\alph*), leftmargin=*]
            \item There exists a graded algebra isomorphism $\rho^\V_N : \V^\DeRham_N \to \gr(\V^\Betti_N)$ uniquely defined by
            \[
                \zeta_N \mapsto [X_0]_0, \quad e_0 \mapsto [X_0^N - 1]_1, \quad e_1 \mapsto [X_1 - 1]_1.
            \]
            \item The graded algebra isomorphism $\rho^\V_N : \V^\DeRham_N \to \gr(\V^\Betti_N)$ restricts to a graded algebra isomorphism $\rho^\W_N : \W^\DeRham_N \to \gr(\W^\Betti_N)$.
        \end{enumerate}
    \end{proposition}
    
    \noindent By Theorem \ref{thm:main_theorem1}, the filtered algebra morphism $\Delta^{\W, \Betti} : \W^\Betti_N \to \W^\Betti_N \otimes \W^\Betti_N$ induces the graded algebra morphism
    \[
        \gr(\Delta^{\W, \Betti}) : \gr(\W^\Betti_N) \to \gr(\W^\Betti_N) \otimes \gr(\W^\Betti_N).
    \]
    The following result is the second main theorem of the paper. It states that the associated graded algebra morphism $\gr(\Delta^{\W, \Betti})$ is in fact the graded algebra morphism $\Delta^{\W, \DeRham}_N$.
    \begin{theorem} \label{thm:maintheorem2}
        The following diagram
        \begin{equation} \label{diag:maintheorem2}
        \begin{tikzcd}
            \W^\DeRham_N \ar[rr, "\Delta^{\W, \DeRham}"] \ar[d, "\rho_N^\W"'] && \W^\DeRham_N \otimes \W^\DeRham_N \ar[d, "\rho_N^\W \otimes \rho_N^\W"] \\
            \gr(\W^\Betti) \ar[rr, "\gr(\Delta^{\W, \Betti})"] && \gr(\W^\Betti) \otimes \gr(\W^\Betti)
        \end{tikzcd}
        \end{equation}
        commutes.
    \end{theorem}
    \noindent Finally, regarding the topological algebra morphism $\widehat{\Delta}^{\W, \Betti}_N$ given in \eqref{eq:conjugation}, Theorems \ref{thm:main_theorem1} and \ref{thm:maintheorem2} motivate the following problem:
    \begin{problem}
        For suitable $a,b \in \mathbb{Z}$, show that the topological algebra morphism $\widehat{\Delta}^{\W, \Betti}_N$ is the completion (w.r.t. the filtration $(\F^m\W^\Betti_N)_{m \in \Z}$) of the algebra morphism $\mathrm{Ad}_{X_1^a Y_1^b} \circ \Delta^{\W, \Betti}$.
    \end{problem}
    
    \section{Compatibility of \texorpdfstring{$\Delta^{\W, \Betti}$}{DeltaWB} with the filtration \texorpdfstring{$(\F^m\V^\Betti_N)_{m \in \Z}$}{(FmVB)m}}
    \noindent In this section, we prove Theorem \ref{thm:main_theorem1}. To do so, we will start with some preparatory results.

    \begin{lemma} \label{lem:FmWB}
    For $m \in \Z_{>0}$, we have
        \begin{enumerate}[label=(\alph*), leftmargin=*]
            \vspace{-2mm}
            \begin{multicols}{2}
            \item \label{FmWB_equality} $\mathcal{F}^m \W^{\Betti}_N = \mathcal{F}^m \V^{\Betti}_N \cap \V^{\Betti}_N (X_1 - 1)$.
            \item \label{FmWB_to_FmVB}$\mathcal{F}^m \W^{\Betti}_N = \mathcal{F}^{m-1} \V^{\Betti}_N (X_1 - 1)$.
            \end{multicols}
            \item \label{FmWB_is_VB_mod}$\F^m\W^\Betti_N$ is a left $\V^\Betti_N$-module.
        \end{enumerate}
    \end{lemma}
    \begin{proof}
        For \ref{FmWB_equality} and \ref{FmWB_to_FmVB}, see \cite[Lemma 3.1.14]{Yad2}. \ref{FmWB_is_VB_mod} follows immediately from \ref{FmWB_to_FmVB}.
    \end{proof}
    
    \begin{lemma} \label{lem:FmWBN}
        For $m \in \Z$, we have
        \[\F^m \W^\Betti_N =
        \begin{cases}
            \W^\Betti_N & \text{if } m \leq 0 \\
            \V^\Betti_N (X_1 - 1) & \text{if } m = 1 \\
            \displaystyle (X_0^N - 1)^{m-1} \K[X_0, X_0^{-1}] (X_1 - 1) + \sum_{k=1}^{m-1} \F^k \W^\Betti_N \cdot \F^{m-k} \W^\Betti_N & \text{if } m \geq 2
        \end{cases}\]
    \end{lemma}
    \begin{proof}
        The result is immediate for $m = 0$; and for $m = 1$, it follows from Lemma~\ref{lem:FmWB}~\ref{FmWB_to_FmVB}. \linebreak
        We now consider the case $m \geq 2$. Since $(\F^n \W^\Betti_N)_{n \in \Z}$ is a decreasing algebra filtration, then
        \begin{equation} \label{eq:Fmsupset1}
            \F^m \W^\Betti_N \supset \sum_{k=1}^{m-1} \F^k \W^\Betti_N \cdot \F^{m-k} \W^\Betti_N.
        \end{equation}
        On the other hand, since $X_0^N-1$ and $X_1-1$ belong to $\mathcal{I}_N$, we obtain the inclusion in the following
        \begin{equation} \label{eq:Fmsubset2}
            \F^m \W^\Betti_N = \F^m \V^\Betti_N \cap \V^\Betti_N (X_1 - 1) \supset (X_0^N - 1)^{m-1} \K[X_0, X_0^{-1}] (X_1 - 1),
        \end{equation}
        and the equality follows from Lemma \ref{lem:FmWB} \ref{FmWB_equality}. From \eqref{eq:Fmsupset1} and \eqref{eq:Fmsubset2}, we obtain the following inclusion
        \[
            \F^m \W^\Betti_N \supset (X_0^N - 1)^{m-1} \K[X_0, X_0^{-1}] (X_1 - 1) + \sum_{k=1}^{m-1} \F^k \W^\Betti_N \cdot \F^{m-k} \W^\Betti_N.
        \]
        Let us now prove the converse.
        The group morphism $F_2 \to \Z$ given by $X_0 \mapsto 1$ and $X_1 \mapsto 0$ admits a section given by $1 \mapsto X_0$. Then $\K{F_2}$ is the direct sum of the image of the section $\K\Z \to \K{F_2}$, which is $\K[X_0, X_0^{-1}]$, and of the kernel of $\K{F_2} \to \K\Z$, which is the two-sided ideal of $\K{F_2}$ generated by $X_1-1$. Let us denote by $\V^\Betti_N (X_1 - 1) \V^\Betti_N$ this ideal\footnote{recall that the algebras $\V^\Betti_N$ and $\K{F_2}$ are equal. In the sequel, we use the former rather that the latter notation for denoting the two-sided ideal generated by $X_1 -1$.}. \newline
        We derive the direct sum decomposition\footnote{where the first summand is a subalgebra of and the second summand is a two-sided ideal}
        \begin{equation*}
            \V^\Betti_N = \K[X_0, X_0^{-1}] \oplus \V^\Betti_N (X_1 - 1) \V^\Betti_N.
        \end{equation*}
        Moreover, since $\V^\Betti_N (X_1 - 1) \V^\Betti_N \subset \mathcal{I}_N = \ker(\K{F_2} \to \K{\mu_N})$, we have 
        \[
            \mathcal{I}_N = \ker\left(\K[X_0, X_0^{-1}] \to \K{\mu_N}\right) \oplus \V^\Betti_N (X_1 - 1) \V^\Betti_N,
        \]
        where $\K[X_0, X_0^{-1}] \to \K{\mu_N}$ is the restriction of $\K{F_2} \to \K{\mu_N}$ to $\K[X_0, X_0^{-1}]$. Therefore,
        \begin{equation} \label{eq:IN_direct_sum}
            \mathcal{I}_N = (X_0^N-1) \K[X_0, X_0^{-1}] \oplus \V^\Betti_N (X_1 - 1) \V^\Betti_N.
        \end{equation}
        Denote by $\mathcal{A}_0 = (X_0^N-1) \K[X_0, X_0^{-1}]$ and $\mathcal{A}_1 = \V^\Betti_N (X_1 - 1) \V^\Betti_N$. Thanks to \eqref{eq:IN_direct_sum}, we obtain
        \begin{equation} \label{eq:INm_direct_sum}
            \mathcal{I}_N^{m-1} = \sum_{\lambda : [\![1,m-1]\!] \to \{0,1\}} \mathcal{A}_{\lambda(1)} \cdots \mathcal{A}_{\lambda(m-1)} = \mathcal{A}_0^{m-1} + \sum_{\substack{\lambda : [\![1,m-1]\!] \to \{0,1\} \\ \lambda \neq \mathbf{0}}} \mathcal{A}_{\lambda(1)} \cdots \mathcal{A}_{\lambda(m-1)},  
        \end{equation}
        where $\mathbf{0} : [\![1,m-1]\!] \to \{0,1\}$ is the zero map. \newline
        Set $X(0):= X_0^N$ and $X(1):= X_1$. Since $\mathcal{A}_i \subset \V^\Betti_N (X(i) - 1) \V^\Betti_N$ (for $i \in \{0, 1\}$), it follows that for any map $\lambda : [\![1,m-1]\!] \to \{0,1\}$, we have
        \begin{equation} \label{eq:prodAlambda_inclusion}
            \mathcal{A}_{\lambda(1)} \cdots \mathcal{A}_{\lambda(m-1)} \subset \V^\Betti_N (X(\lambda(1)) - 1) \V^\Betti_N \cdots \V^\Betti_N (X(\lambda(m-1)) - 1) \V^\Betti_N.
        \end{equation}
        Combining equality \eqref{eq:INm_direct_sum}, inclusion \eqref{eq:prodAlambda_inclusion} for $\lambda \neq \mathbf{0}$,  and the equality $\mathcal{A}_0^{m-1} = (X_0^N-1)^{m-1} \K[X_0, X_0^{-1}]$, we obtain
        \[
            \mathcal{I}_N^{m-1} \subset (X(0) - 1)^{m-1} \K[X_0, X_0^{-1}] + \sum_{\substack{\lambda : [\![1,m-1]\!] \to \{0,1\} \\ \lambda \neq \mathbf{0}}} \V^\Betti_N (X(\lambda(1)) - 1) \V^\Betti_N \cdots \V^\Betti_N (X(\lambda(m-1)) - 1) \V^\Betti_N.
        \]
        Since $X(i) - 1 \in \mathcal{I}_N$ (for $i \in \{0, 1\}$), the right hand side of this inclusion is contained in $\mathcal{I}_N^{m-1}$, therefore
        \begin{equation} \label{eq:INm_equality}
            \mathcal{I}_N^{m-1} = (X(0) - 1)^{m-1} \K[X_0, X_0^{-1}] + \sum_{\substack{\lambda : [\![1,m-1]\!] \to \{0,1\} \\ \lambda \neq \mathbf{0}}} \V^\Betti_N (X(\lambda(1)) - 1) \V^\Betti_N \cdots \V^\Betti_N (X(\lambda(m-1)) - 1) \V^\Betti_N.
        \end{equation}
        Finally,
        \begin{align*}
            \F^m\W^\Betti_N = & \mathcal{I}_N^{m-1} (X(1) - 1) \\[2mm]
            = & (X(0) - 1)^{m-1} \K[X_0, X_0^{-1}] (X(1) - 1) \\
            & + \sum_{\substack{\lambda : [\![1,m-1]\!] \to \{0,1\} \\ \lambda \neq \mathbf{0}}} \V^\Betti_N (X(\lambda(1)) - 1) \cdots \V^\Betti_N (X(\lambda(m-1)) - 1) \V^\Betti_N (X(1) - 1) \\[2mm]
            = & (X(0) - 1)^{m-1} \K[X_0, X_0^{-1}] (X(1) - 1) \\
            & + \sum_{\lambda \in \Lambda_m} \V^\Betti_N (X(\lambda(1)) - 1) \cdots \V^\Betti_N (X(\lambda(m-1)) - 1) \V^\Betti_N (X(\lambda(m)) - 1) \\[2mm]
            = & (X(0) - 1)^{m-1} \K[X_0, X_0^{-1}] (X(1) - 1) \\
            & + \sum_{j \geq 2} \sum_{(k_1, \dots, k_j) \in \mathcal{K}^{(j)}_m} \left(\V^\Betti_N (X(0) - 1)\right)^{k_1-1} \V^\Betti_N (X(1) - 1) \left(\V^\Betti_N (X(0) - 1)\right)^{k_2-k_1-1} \\[-3mm]
            & \hspace{32mm} \V^\Betti_N (X(1) - 1) \cdots \left(\V^\Betti_N (X(0) - 1)\right)^{k_j-k_{j-1}-1} \V^\Betti_N (X(1) - 1) \\[2mm]
            \subset & (X(0) - 1)^{m-1} \K[X_0, X_0^{-1}] (X(1) - 1) \\
            & + \sum_{j \geq 2} \sum_{(k_1, \dots k_j) \in \mathcal{K}^{(j)}_m} \F^{k_1} \W^\Betti_N \cdot \F^{k_2-k_1} \W^\Betti_N \cdots \F^{k_j-k_{j-1}} \W^\Betti_N \\[2mm] 
            \subset & (X(0) - 1)^{m-1} \K[X_0, X_0^{-1}] (X(1) - 1) + \sum_{k = 1}^{m-1} \F^k \W^\Betti_N \cdot \F^{m-k} \W^\Betti_N,
        \end{align*}
        where the first equality follows from Lemma \ref{lem:FmWB} \ref{FmWB_to_FmVB} and the second one from \eqref{eq:INm_equality}. In the third equality one denotes
        \[
            \Lambda_m := \{\lambda : [\![1,m]\!] \to \{0,1\} \mid \lambda(m) = 1, \ \lambda_{|[\![1,m-1]\!]} \neq \mathbf{0}\}
        \]
        and the equality then follows immediately. In the fourth equality one denotes 
        \[
            \mathcal{K}^{(j)}_m := \{ (k_1, \dots, k_j) \mid 1 \leq k_1 < \cdots < k_{j-1} < k_j = m \},
        \]
        one also uses Convention \ref{convention_star} for the definition of $\left(\V^\Betti_N (X(0) - 1)\right)^k$ (for any integer $k \geq 1$); and the equality is induced by the bijection
        \[
            \Lambda_m \simeq \bigsqcup_{j \geq 2} \mathcal{K}^{(j)}_m, \quad \lambda \mapsto \lambda^{-1}(\{0\}).
        \]
        The first inclusion follows from the fact $(\V^\Betti_N (X(0) - 1))^{k-1} \V^\Betti_N (X(1) - 1) \subset \F^k \W^\Betti_N$ (for any integer $k \geq 1$); and the last inclusion from the fact that $(\F^m \W^\Betti_N)_{m \in \Z}$ is a decreasing filtration and therefore
        \[
            \F^{k_2-k_1} \W^\Betti_N \cdot \cdots \cdot \F^{k_j-k_{j-1}} \W^\Betti_N \subset \F^{k_j-k_1} \W^\Betti_N = \F^{m-k_1} \W^\Betti_N.
        \]
    \end{proof}
    
    \begin{lemma} \label{lem:DeltaWNFmWBN}
        For any integer $m \geq 2$, we have
        \[
            \Delta^{\W, \Betti}\left((X_0^N - 1)^{m-1} \K[X_0, X_0^{-1}] (X_1 - 1)\right) \subset \F^m (\W^\Betti_N \otimes \W^\Betti_N) 
        \]
    \end{lemma}
    \begin{proof}
        Let $P(X_0, X_0^{-1}) \in \K[X_0, X_0^{-1}]$. We have
        \begin{align}
            &\Delta^{\W, \Betti}\big((X_0^N - 1)^{m-1} P(X_0, X_0^{-1}) (X_1 - 1)\big) \label{eq:DeltaWBettielt1}\\
            & = (X_0^N - 1)^{m-1} P(X_0, X_0^{-1}) (X_1 - 1) + (Y_0^N - 1)^{m-1} P(Y_0, Y_0^{-1}) (Y_1 - 1) \notag \\
            & - \frac{(X_0^N - 1)^{m-1} P(X_0, X_0^{-1}) Y_0 - (Y_0^N - 1)^{m-1} P(Y_0, Y_0^{-1}) X_0}{X_0 - Y_0} (X_1 - 1) (Y_1 - 1), \notag
        \end{align}
        where $\frac{(X_0^N - 1)^{m-1} P(X_0, X_0^{-1}) Y_0 - (Y_0^N - 1)^{m-1} P(Y_0, Y_0^{-1}) X_0}{X_0 - Y_0}$ is the polynomial $F(X_0, X_0^{-1}, Y_0, Y_0^{-1}) \in \K[X_0, X_0^{-1}, Y_0, Y_0^{-1}]$ such that
        \[
            (X_0 - Y_0) F(X_0, X_0^{-1}, Y_0, Y_0^{-1}) = (X_0^N - 1)^{m-1} P(X_0, X_0^{-1}) Y_0 - (Y_0^N - 1)^{m-1} P(Y_0, Y_0^{-1}) X_0.
        \]
        Next, we have
        \begin{align*}
            & \frac{(X_0^N - 1)^{m-1} P(X_0, X_0^{-1}) Y_0 - (Y_0^N - 1)^{m-1} P(Y_0, Y_0^{-1}) X_0}{X_0 - Y_0} = - (X_0^N - 1)^{m-1} P(X_0, X_0^{-1}) \\
            & - (Y_0^N - 1)^{m-1} P(Y_0, Y_0^{-1}) + \frac{(X_0^N - 1)^{m-1} P(X_0, X_0^{-1}) X_0 - (Y_0^N - 1)^{m-1} P(Y_0, Y_0^{-1}) Y_0}{X_0 - Y_0} \\[1em]
            & = - (X_0^N - 1)^{m-1} P(X_0, X_0^{-1}) - (Y_0^N - 1)^{m-1} P(Y_0, Y_0^{-1}) \\[1pt]
            & + \frac{\mbox{\small$(X_0^N - 1)^{m-1} \big(P(X_0, X_0^{-1}) X_0 - P(Y_0, Y_0^{-1}) Y_0\big)$}}{X_0 - Y_0} + \frac{\mbox{\small$\left((X_0^N - 1)^{m-1} - (Y_0^N - 1)^{m-1}\right) P(Y_0, Y_0^{-1}) Y_0$}}{X_0 - Y_0}
        \end{align*}
        Denote by
        \begin{align*}
            A(X_0, X_0^{-1}, Y_0, Y_0^{-1}) & := - (X_0^N - 1)^{m-1} P(X_0, X_0^{-1}) - (Y_0^N - 1)^{m-1} P(Y_0, Y_0^{-1}), \\[1em]
            B(X_0, X_0^{-1}, Y_0, Y_0^{-1}) & := \frac{(X_0^N - 1)^{m-1} \big(P(X_0, X_0^{-1}) X_0 - P(Y_0, Y_0^{-1}) Y_0\big)}{X_0 - Y_0}, \\[1em]
            C(X_0, X_0^{-1}, Y_0, Y_0^{-1}) & := \frac{\left((X_0^N - 1)^{m-1} - (Y_0^N - 1)^{m-1}\right) P(Y_0, Y_0^{-1}) Y_0}{X_0 - Y_0}.
        \end{align*}
        Thanks to this, we obtain from equality \eqref{eq:DeltaWBettielt1} the following identity
        \begin{align}
            & \Delta^{\W, \Betti}\big((X_0^N - 1)^{m-1} P(X_0, X_0^{-1}) (X_1 - 1)\big) = (X_0^N - 1)^{m-1} P(X_0, X_0^{-1}) (X_1 - 1) \label{eq:DeltaWBelt} \\
            & + (Y_0^N - 1)^{m-1} P(Y_0, Y_0^{-1}) (Y_1 - 1) - A(X_0, X_0^{-1}, Y_0, Y_0^{-1}) (X_1 - 1) (Y_1 -1) \notag \\
            & - B(X_0, X_0^{-1}, Y_0, Y_0^{-1}) (X_1 - 1) (Y_1 -1) - C(X_0, X_0^{-1}, Y_0, Y_0^{-1}) (X_1 - 1) (Y_1 -1). \notag
        \end{align}
        Since $X_0^N-1, X_1-1 \in \mathcal{I}_N$, we have
        \begin{equation} \label{eq:elt_of_FmWBN}
            (X_0^N - 1)^{m-1} P(X_0, X_0^{-1}) (X_1 - 1) \in \F^m \V^\Betti_N \cap \W^\Betti = \F^m \W^\Betti_N,
        \end{equation}
        Then the statement \eqref{eq:elt_of_FmWBN} implies that
        \[
            (X_0^N - 1)^{m-1} P(X_0, X_0^{-1}) (X_1 - 1) \in \F^m \W^\Betti_N \otimes 1 \subset \F^m (\W^\Betti_N \otimes \W^\Betti_N),
        \]
        and
        \[
            (Y_0^N - 1)^{m-1} P(Y_0, Y_0^{-1}) (Y_1 - 1) \in 1 \otimes \F^m \W^\Betti_N \subset \F^m (\W^\Betti_N \otimes \W^\Betti_N).
        \]
        On the other hand, we have
        \begin{align}
            A(X_0, X_0^{-1}, Y_0, Y_0^{-1}) (X_1 - 1) (Y_1 -1) & = - P(X_0, X_0^{-1}) (X_0^N - 1)^{m-1} (X_1 - 1) (Y_1 -1) \label{A_in_Fm1} \\
            & - P(Y_0, Y_0^{-1}) (Y_0^N - 1)^{m-1} (Y_1 -1) (X_1 - 1) \notag \\
            & \in \F^{m+1} \left(\W^\Betti_N \otimes \W^\Betti_N\right), \notag
        \end{align}
        where the ``$\in$'' claim follows from the fact that $(X_0^N - 1)^{m-1} (X_1 - 1) (Y_1 -1) \in \F^m \W^\Betti_N \otimes \F^1\W^\Betti$ and that $\F^m \W^\Betti_N \otimes \F^1\W^\Betti$ is a left $(\V^\Betti_N \otimes \V^\Betti_N)$-module, which implies
        \[
            - P(X_0, X_0^{-1}) (X_0^N - 1)^{m-1} (X_1 - 1) (Y_1 -1) \in \F^m \W^\Betti_N \otimes \F^1\W^\Betti.
        \]
        Swapping between $X$ and $Y$ enables us to apply the same argument to show that
        \[
            - P(Y_0, Y_0^{-1}) (Y_0^N - 1)^{m-1} (Y_1 -1) (X_1 - 1) \in \F^1\W^\Betti \otimes \F^m \W^\Betti_N.
        \]
        Moreover, we have
        \begin{align}
            & B(X_0, X_0^{-1}, Y_0, Y_0^{-1}) (X_1 - 1) (Y_1 -1) \label{B_in_Fm1} \\
            & = \frac{P(X_0, X_0^{-1}) X_0 - P(Y_0, Y_0^{-1}) Y_0}{X_0 - Y_0} (X_0^N - 1)^{m-1} (X_1 - 1) (Y_1 -1) \notag \\
            & \in \F^m \W^\Betti_N \otimes \F^1 \W^\Betti_N \subset \F^{m+1} \left(\W^\Betti_N \otimes \W^\Betti_N\right), \notag
        \end{align}
        where the ``$\in$'' claim follows from the fact that $(X_0^N - 1)^{m-1} (X_1 - 1) (Y_1 -1) \in \F^m \W^\Betti_N \otimes \F^1\W^\Betti$ and that $\F^m \W^\Betti_N \otimes \F^1\W^\Betti$ is a left $(\V^\Betti_N \otimes \V^\Betti_N)$-module. \newline
        Moreover, we have
        \begin{align}
            & C(X_0, X_0^{-1}, Y_0, Y_0^{-1}) (X_1 - 1) (Y_1 -1) \label{C_in_Fm} \\
            & = P(Y_0, Y_0^{-1}) Y_0 \frac{X_0^N - Y_0^N}{X_0 - Y_0} \frac{(X_0^N - 1)^{m-1} - (Y_0^N - 1)^{m-1}}{X_0^N - Y_0^N} (X_1 - 1) (Y_1 -1) \notag \\
            & = P(Y_0, Y_0^{-1}) Y_0 \left(\sum_{k=0}^{N-1} X_0^k Y_0^{N-1-k}\right) \sum_{l=0}^{m-2} \underbrace{(X_0^N -1)^l (X_1 - 1)}_{\in \F^{l+1} \W^\Betti_N \otimes 1} \underbrace{(Y_0^N -1)^{m-2-l} (Y_1 - 1)}_{\in 1 \otimes \F^{m-1-l} \W^\Betti_N} \notag \\
            & \in \F^m (\W^\Betti_N \otimes \W^\Betti_N). \notag
        \end{align}
        Therefore, it follows from identity \eqref{eq:DeltaWBelt} that
        \begin{equation*}
            \Delta^{\W, \Betti}\big((X_0^N - 1)^{m-1} P(X_0, X_0^{-1}) (X_1 - 1)\big) \in \F^m (\W^\Betti_N \otimes \W^\Betti_N). 
        \end{equation*}
    \end{proof}

    \begin{proof}[Proof of Theorem \ref{thm:main_theorem1}]
        If $m \leq 0$, the result is immediate. Let us assume that $m \geq 1$. We will proceed with the proof by induction on $m$. \newline
        For $m=1$, denote by $\varepsilon : \W^\Betti_N \to \K$ the counit of the bialgebra $(\W^\Betti_N, \Delta^{\W, \Betti})$. We have
        \[
            \Delta^{\W, \Betti}(\F^1\W^\Betti_N) = \Delta^{\W, \Betti}(\ker(\varepsilon)) \subset \ker(\varepsilon \otimes \varepsilon) = \F^1 (\W^\Betti_N \otimes \W^\Betti_N),  
        \]
        where the first equality follows from the identity $\F^1 \W^\Betti_N = \ker(\varepsilon)$; the second equality from the counit identity $\Delta^{\W, \Betti} \circ \varepsilon = (\varepsilon \otimes \varepsilon) \circ \Delta^{\W, \Betti}$; and the third equality from the identity \linebreak $\F^1 (\W^\Betti_N \otimes \W^\Betti_N) = \ker(\varepsilon \otimes \varepsilon)$. \newline
        Suppose now that the statement is true until $m-1$. We have
        \begin{align*}
            & \Delta^{\W, \Betti}(\F^m\W^\Betti_N) = \Delta^{\W, \Betti}\left((X_0^N - 1)^{m-1} \K[X_0, X_0^{-1}] (X_1 - 1) + \sum_{k=1}^{m-1} \F^k \W^\Betti_N \cdot \F^{m-k} \W^\Betti_N\right) \\
            & \subset \Delta^{\W, \Betti}\left((X_0^N - 1)^{m-1} \K[X_0, X_0^{-1}] (X_1 - 1)\right) + \sum_{k=1}^{m-1} \Delta^{\W, \Betti}\left(\F^k \W^\Betti_N \right) \cdot \Delta^{\W, \Betti}\left(\F^{m-k} \W^\Betti_N\right) \\
            & \subset \F^m (\W^\Betti_N \otimes \W^\Betti_N) + \sum_{k=1}^{m-1} \F^k (\W^\Betti_N \otimes \W^\Betti_N) \cdot \F^{m-k} (\W^\Betti_N \otimes \W^\Betti_N) \\
            & \subset \F^m (\W^\Betti_N \otimes \W^\Betti_N),
        \end{align*}
        where the equality follows from Lemma \ref{lem:FmWBN}, the first inclusion follows by linearity of $\Delta^{\W, \Betti}$ and compatibility with the product; the second inclusion from Lemma \ref{lem:DeltaWNFmWBN} and induction hypothesis; and the last inclusion from the fact that $\left(\F^m (\W^\Betti_N \otimes \W^\Betti_N)\right)_{m \in \Z}$ is an algebra filtration, which follows from the fact that $(\F^m\W^\Betti_N)_{m \in \Z}$ is an algebra filtration.
    \end{proof}

    \section{Computation of \texorpdfstring{$\gr(\Delta^{\W, \Betti})$}{gr(DeltaWB)}}
    \noindent In this section, we prove Theorem \ref{thm:maintheorem2}. 
    \begin{proof}[Proof of Theorem \ref{thm:maintheorem2}]
        Let us prove that diagram \eqref{diag:maintheorem2} of graded algebra morphisms commutes for any degree $m \geq 1$. \newline
        For $a \in [\![0, N-1]\!]$, $z_{m, \zeta_N^a}$ is a degree $m$ element of $\W^\DeRham_N$ and we have
        \begin{equation} \label{eq:rho_z}
            \rho_N^\W(z_{m, \zeta_N^a}) = [(X_0^N-1)^{m-1} X_0^a (1-X_1)]_m.
        \end{equation}
        Recall that
        \[
            \Delta^{\W, \DeRham}_N (z_{m, \zeta_N^a}) = z_{m, \zeta_N^a} \otimes 1 + 1 \otimes z_{m, \zeta_N^a} + \sum_{\substack{1 \leq k \leq m-1\\ 0 \leq b \leq N-1}} z_{k, \zeta_N^b} \otimes z_{m-k, \zeta_N^{a-b}}.
        \]
        Therefore, we obtain
        \begin{multline*}
            \hspace{-4mm}\left(\rho_N^\W \otimes \rho_N^\W\right) \circ \Delta^{\W, \DeRham}_N (z_{m, \zeta_N^a}) \hfill \\
            = \Big[(X_0^N-1)^{m-1} X_0^a (1 - X_1) + (Y_0^N-1)^{m-1} Y_0^a (1 - Y_1) \hfill \\ \hfill + \sum_{\substack{1 \leq k \leq m-1\\ 0 \leq b \leq N-1}} (X_0^N-1)^{k-1} X_0^b (1 - X_1) (Y_0^N-1)^{m-k-1} Y_0^{a-b} (1 - Y_1)\Big]_m
        \end{multline*}
        On the other hand, by taking $P(X_0, X_0^{-1}) = X_0^a$ in \eqref{eq:DeltaWBelt}, we obtain that
        \begin{multline*}
            \Delta^{\W, \Betti}\left((X_0^N-1)^{m-1} X_0^a (1-X_1)\right) = (X_0^N - 1)^{m-1} X_0^a (1 - X_1) + (Y_0^N - 1)^{m-1} Y_0^a (1- Y_1) \hfill \\
            \hfill + \widetilde{A}(X_0, Y_0) (1 - X_1) (1 - Y_1) + \widetilde{B}(X_0, Y_0) (1 - X_1) (1 - Y_1) + \widetilde{C}(X_0, Y_0) (1 - X_1) (1 - Y_1),
        \end{multline*}
        where
        \begin{align*}
            \widetilde{A}(X_0, Y_0) & := - (X_0^N - 1)^{m-1} X_0^a - (Y_0^N - 1)^{m-1} Y_0^a, \\[1em]
            \widetilde{B}(X_0, Y_0) & := (X_0^N - 1)^{m-1}\frac{X_0^{a+1} - Y_0^{a+1}}{X_0 - Y_0}, \\[1em]
            \widetilde{C}(X_0, Y_0) & := \frac{(X_0^N - 1)^{m-1} - (Y_0^N - 1)^{m-1}}{X_0 - Y_0} Y_0^{a+1}.
        \end{align*}
        Thanks to \eqref{A_in_Fm1}, \eqref{B_in_Fm1} and \eqref{C_in_Fm}, it follows that
        \begin{align*}
            \widetilde{A}(X_0, Y_0) (1 - X_1) (1 - Y_1) \in \F^{m+1} (\W^\Betti_N \otimes \W^\Betti_N), \\[1em]
            \widetilde{B}(X_0, Y_0) (1 - X_1) (1 - Y_1) \in \F^{m+1} (\W^\Betti_N \otimes \W^\Betti_N), \\[1em]
            \widetilde{C}(X_0, Y_0) (1 - X_1) (1 - Y_1) \in \F^m (\W^\Betti_N \otimes \W^\Betti_N).
        \end{align*}
        Therefore, thanks to equality \eqref{eq:rho_z}, we obtain
        \begin{align*}
            & \gr(\Delta^{\W, \Betti}) \circ \rho_N^\W(z_{m, \zeta_N^a}) \\
            & = \left[(X_0^N-1)^{m-1} X_0^a (1 - X_1) + (Y_0^N-1)^{m-1} Y_0^a (1 - Y_1) + \widetilde{C}(X_0, Y_0) (1 - X_1) (1 - Y_1)\right]_m.
        \end{align*}
        One checks that
        \begin{align*}
            \widetilde{C}(X_0, Y_0) & = \left(\sum_{k=1}^{m-1} (X_0^N-1)^{k-1} (Y_0^N-1)^{m-k-1} \right) \left(\sum_{b=0}^{N-1} X_0^b Y_0^{N-1-b} \right) Y_0^{a+1} \\
            & = \sum_{\substack{1 \leq k \leq m-1\\ 0 \leq b \leq N-1}} (X_0^N-1)^{k-1} X_0^b (Y_0^N-1)^{m-k-1} Y_0^{N+a-b}.
        \end{align*}
        Finally,
        \begin{multline*}
            \hspace{-3mm}\gr(\Delta^{\W, \Betti}) \circ \rho_N^\W(z_{m, \zeta_N^a}) \hfill \\
            = \Big[(X_0^N-1)^{m-1} X_0^a (1 - X_1) + (Y_0^N-1)^{m-1} Y_0^a (1 - Y_1) \hfill \\ \hfill + \sum_{\substack{1 \leq k \leq m-1\\ 0 \leq b \leq N-1}} (X_0^N-1)^{k-1} X_0^b (1 - X_1) (Y_0^N-1)^{m-k-1} Y_0^{a-b} (1 - Y_1)\Big]_m.
        \end{multline*}
        This concludes the proof.
    \end{proof}

    \begin{acknowledgments}
        This project was partially supported by first author's ANR grant Project HighAGT ANR20-CE40-0016 and second author's JSPS KAKENHI Grant 23KF0230.
    \end{acknowledgments}
    
    \bibliographystyle{abstract}
    \bibliography{main}
\end{document}